\documentclass[twoside,12pt, leqno]{article}
\usepackage{amsmath,amscd,amsthm,amssymb,amsxtra,latexsym,epsfig,epic,graphics}
\usepackage[matrix,arrow,curve]{xy}
\usepackage{graphicx}
\usepackage{diagrams}
\usepackage[breaklinks,bookmarksopen,bookmarksnumbered]{hyperref}
\hypersetup{colorlinks=true,backref=true,citecolor=blue}

\def\antiddot{\mathinner{\mkern1mu\raise1pt\vbox{\kern7pt\hbox{.}}\mkern2mu
        \raise4pt\hbox{.}\mkern2mu\raise7pt\hbox{.}\mkern1mu}}

\newcommand{\FF}{{\mathbb F}}

\newcommand{\PP}{{\mathbb P}}
\newcommand{\QQ}{{\mathbb Q}}

\newcommand{\s}{\mathcal}

\newcommand{\cO}{{\s O}}
\newcommand{\cF}{{\s F}}

\newcommand{\punkt}{\hspace{-.3ex}\raise.15ex\hbox to1ex{\Huge.}}

\def \fix#1 {{\hfill\break \bf (( #1 ))\hfill\break}}

\DeclareMathOperator{\Spec}{Spec}

\DeclareMathOperator{\Hom}{Hom}

\newtheorem{theorem}{Theorem}[section]
\newtheorem{lemma}[theorem]{Lemma}
\newtheorem{proposition}[theorem]{Proposition}
\newtheorem{corollary}[theorem]{Corollary}

\theoremstyle{definition}

\newtheorem{remark}[theorem]{Remark}

\def\e{{\epsilon}}

\def\AA{{\mathbb A}}

\def\PP{{\mathbb P}}
\def\P{{\mathbb P}}
\def\QQ{{\mathbb Q}}
\def\FF{{\mathbb F}}

\def\cF{{\cal F}}

\def\cH{{\cal H}}

\def\fix#1{{\bf ***} #1 {\bf ***}}

\def\CI{{\mathcal I}}
\def\CH{{\mathcal H}}
\def\CCH{{\mathcal {HC}}}

\def\CO{{\mathcal O}}
\def\CT{{\mathcal T}}
\def\CHom{{\mathcal Hom}}
\def\Spec{{{\rm Spec}\,}}
\def\cone{{{\rm cone}\,}}

\makeatletter
\def\Ddots{\mathinner{\mkern1mu\raise\p@
\vbox{\kern7\p@\hbox{.}}\mkern2mu
\raise4\p@\hbox{.}\mkern2mu\raise7\p@\hbox{.}\mkern1mu}}
\makeatother

\makeatletter
\def\Ddots{\mathinner{\mkern1mu\raise\p@
\vbox{\kern7\p@\hbox{.}}\mkern2mu
\raise4\p@\hbox{.}\mkern2mu\raise7\p@\hbox{.}\mkern1mu}}
\makeatother

\usepackage{times}
\newdimen\x \x=12pt

\usepackage{color}

\date{}
\title{Twenty Points in $\PP^{3}$}
\author{David Eisenbud, Robin Hartshorne, and Frank-Olaf Schreyer
\footnote{This paper reports on work done during the Commutative Algebra Program, 2012-13,
at MSRI. We are grateful
to MSRI for providing such an exciting environment, where a chance meeting
led to the progress described here. The first author is also grateful to the 
National Science Foundation for partial support during this period.}}

\begin{document}

\maketitle

\begin{abstract}
Using the possibility of computationally determining points on a finite cover of a unirational variety over a finite field, we determine all possibilities for direct Gorenstein linkages between general sets of points in $\PP^{3}$
over an algebraically closed field of characteristic 0. As a consequence we show that a general set of $d$ points is glicci (that is, in the Gorenstein linkage
class of a complete intersection) if $d\leq 33$ or $d=37,38$. Computer algebra plays an essential role in the proof. The case of
20 points had been an outstanding problem in the area for a dozen years~\cite{Hartshorne2001}.
\end{abstract}

%\medskip\noindent Issues:
%\begin{enumerate}
%%\item What's the situation for artinian (maybe graded artinian) rings -- are they known to be glicci?
%%\item find reference or proof of the binomial coefficient formula. (wrote to Gessel)
%%\item remark that we don't find an explicit example
%\item wrote to Cameron for a reference to derangement limit
%\end{enumerate}

\section*{Introduction} The theory of liaison (linkage) is a powerful tool in the theory of curves in $\P^{3}$ with applications, for example, to the question of the unirationality of the moduli spaces of curves (for example~\cite{Chang-Ran,Verra,Schreyer}. One says that two curves $C,D\subset \P^{3}$ (say, reduced and without common components) are directly linked if their union is a complete intersection, and \emph{evenly linked} if there is a chain of curves $C= C_{0}, C_{1}, \dots, C_{2m}=D$ such that $C_{i}$ is directly linked to $C_{i+1}$ for all $i$. The first step in theory is the result of Gaeta that any two arithmetically Cohen-Macaulay curves are evenly linked, and in particular are \emph{in the linkage class of a complete intersection}, usually written \emph{licci}. Much later Rao (\cite{Rao1978}) showed that even linkage classes are in bijection with graded modules of finite length up to shift, leading to an avalanche of results (reported, for example in~\cite{Migliore1998} and \cite{Martin-Deschamps-Perrin}). However, in codimension $>2$ linkage yields an equivalence relation that seems to be very fine, and thus not so useful; for example the scheme consisting of the 4 coordinate points in $\PP^{3}$ is not licci. 

A fundamental paper of Peskine and Szpiro~\cite{PS1974} laid the modern foundation for the theory of linkage. They observed that some of the duality used in liaison held more generally in a Gorenstein context, and Schenzel \cite{Schenzel} introduced a full theory of Gorenstein liaison. We say that two schemes $X,Y\subset \PP^{n}$ that are reduced and without common components are \emph{directly Gorenstein-linked} if their union  is  arithmetically Gorenstein (for general subschemes the right definition is that $\CI_{G}:\CI_{X}=\CI_{Y}$ and $\CI_{G}:\CI_{Y}=\CI_{X}$). We define \emph{Gorenstein linkage} to be the equivalence relation generated by this notion. This does not change the codimension 2 theory, since, by a result of Serre (\cite{Serre}) every Gorenstein scheme of codimension 2 is a complete intersection. Moreover, the behavior of Gorenstein linkage in higher codimensions seems closer to that in codimension 2. For example, it is not hard to show that  a set of 4 points of $\PP^{3}$ is linked via Gorenstein schemes to a complete intersection---that is, in now-standard terminology, it is \emph{glicci} (see for example \cite{KMMNP}.)

For the theory in higher codimension to be similar to the codimension 2 theory, one hopes that every arithmetically Cohen-Macaulay scheme 
 is glicci, and, in particular: \emph{every} finite set of points in $\PP^{3}$ should be glicci. This was verified by the second author in 2001(see~\cite{Hartshorne2001}) for general sets of $d$ points in $\PP^{3}$ with $d<20$, and he proposed the case of 20 general points in $\PP^{3}$ as a ``first candidate counterexample''. The question of of whether 20 general points in $\PP^{3}$ is glicci has remained open since then.

\begin{theorem}\label{main theorem}
Over an algebraically closed field of characteristic zero, a scheme consisting of $d$ general points in $\PP^{3}$ is glicci when $d\leq 33$ and also when $d=37$ or $d=38$.
\end{theorem}

Further, we determine all pairs of numbers $d,e$ such that there exist ``bi-domi\-nant'' direct Goren\-stein linkage correspondences between the smoothing components (that is, the components containing reduced sets of points) of the Hilbert schemes of degree $d$ subschemes and degree $e$ finite subschemes of $\PP^{3}$---see Section~\ref{section:sketch} for a precise statement. All such bi-dominant correspondences are indicated by the edges in the graph shown in Figure~\ref{Fig1-1}. Our approach makes essential use of computation, done in {\it Macaulay2}~\cite{M2} by the package \href{http://www.math.uni-sb.de/ag/schreyer/home/computeralgebra.htm}{\emph{GlicciPointsInP3}} \cite{ES-glicci}. It passes by way of characteristic $p>0$, and we get the same results in all the characteristics we have tested. 

\begin{figure}
\small 
\hskip .25in {\includegraphics[height=450pt] {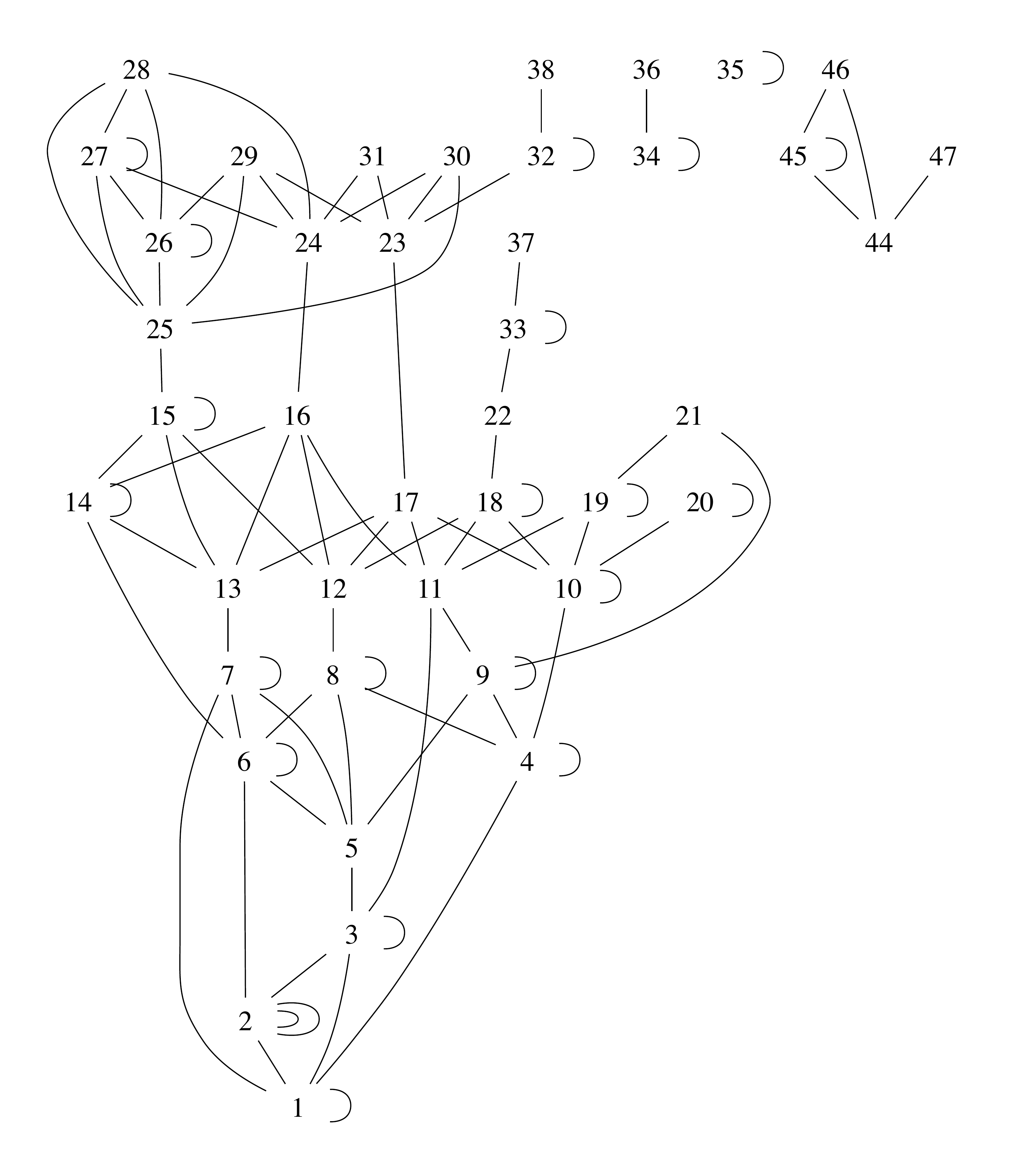}}
\caption{The graph shows all bi-dominant Gorenstein direct linkage correspondences for collections of points in $\PP^{3}$ \label{Fig1-1}
over an algebraically closed field of characteristic zero. 
A vertex $d$ represents a set of $d$ points, and an edge\ %\hfill\break 
$d$---$e$ represents linkage by an arithmetically Gorenstein set of $d+e$ points with a particular $h$-vector. Thus two edges are shown where two different h-vectors are possible.}
\vskip15pt
\end{figure}

Paradoxically, though our method proves that almost all sets of points of the given degrees is glicci, and uses computation, it is not constructive:  we prove that a general set of 20 points in $\PP^{3}$ is linked to a set of 10 points by a Gorenstein scheme of length 30. But if you give us a set of 20 points, we have no way of producing a Gorenstein scheme of length 30 containing it. 

As far as we know, our paper is the first to make explicit use of linkages by Gorenstein that are
not divisors of the form $mH-K$ on some arithmetically Cohen-Macaulay scheme of one larger dimension (see Section 5, below).
\medskip

\noindent{\bf Thanks:} We are grateful to Peter Cameron, Persi Diaconis, Ira Gessel, Robert Guralnick,  Brendan Hassett and Yuri Tschinkel for helping us to understand various aspects of this paper.

\section{Basic definitions and outline of the argument}\label{section:sketch}

For simplicity, we work throughout this paper over a perfect (but not necessarily algebraically closed) field $k$. Given a closed subscheme $X\subset \PP^{n}$,
we write $S_{X}$ for the homogeneous coordinate ring of $X$ and $\cO_{X}$ for the structure sheaf of $X$. Similarly, we write $I_{X}$ and $\CI_{X}$ for the homogeneous ideal and the ideal sheaf of $X$.

 We write $\CH_{X}$ for the union of the components of the Hilbert scheme of subschemes of $\PP^{n}$ that contain $X$; and we write $\CCH_{X}$ for the union of the components of the Hilbert scheme of cones in $\AA^{n+1}$
that contain the cone over $X$ (see \cite{Haiman-Sturmfels}). For example, if $X$ is a reduced set of points then $\CH_{X}$ is smooth at $X$ of dimension $nd$,
but $\CCH_{X}$ may be more complicated.

The reason for considering these cones is the following. While, for a set of general points $X\subset \PP^{n}$, the deformation theory and Hilbert schemes $\CH_{X}$ and $\CCH_{X}$ are naturally isomorphic, this is not so for other subschemes. Also, if $G$ is an arithmetically Gorenstein set of points in $\PP^{n}$, then deformations of $G$ will not in general be arithmetically Gorenstein. On the other hand, the deformations of the cone over $G$ in $\CCH_{G}$ are again cones over arithmetically Gorenstein schemes, and this is the way one normally defines the scheme structure (called PGor$(H)$
in \cite{Kleppe1998}) on the subset of the Hilbert scheme of points in $\PP^{n}$ consisting of arithmetically Gorenstein subschemes. In particular, the tangent space to $\CCH_{G}$ at the point corresponding to $G$ is $\Hom_{S}(I_{G}, S/I_{G})_{0}$.

We say that a finite scheme of degree $d$ in $\PP^{n}$ has \emph{generic Hilbert function}
if the maps $H^{0}(\CO_{\PP^{n}}(d)) \to H^{0}(\CO_{X}(d))$
are either injective or surjective for each $d$. This is the case, for example, when $X$ consists of $d$ general reduced points. In any case, such a generic Hilbert function is determined by $d$ alone. Moreover, any nearby set of points will also have generic Hilbert function, so the Hilbert scheme $\CH_{X}$ is naturally isomorphic, in a neighborhood of $X$, to the Hilbert scheme $\CCH_{X}$,
and in particular the latter is also smooth and of dimension $nd$ at $X$.

By~\cite{BE}, any arithmetically Gorenstein scheme $G$ of codimension 3 in $\PP^{n}$ is defined by the Pfaffians of a skew symmetric matrix of homogeneous forms. From this description and Macaulay's growth conditions Stanley~\cite{Stanley1978} derived a characterization of the possible Hilbert functions of such schemes, and \cite{Diesel,Kleppe1998} showed that the family $\CCH_{G}$ is smooth and irreducible. To simplify our discussion, we will use the term \emph{Gorenstein Hilbert function} in this paper to refer to a Hilbert function of some arithmetically Gorenstein scheme of codimension 3 in $\PP^{3}$.

Given two positive integers $d,e$, we ask whether there exists a a zero-dimen\-sional arithmetically Gorenstein scheme $G$ of degree $d+e$ in $\P^{3}$ that ``could'' provide a direct Gorenstein linkage between
a set $X$ of $d$  points with generic Hilbert function and a set $Y$ of $e$ points with generic Hilbert function. It turns out that there is only a finite number of possibilities for the Hilbert function of such a scheme, and we can list them. 

If $G$ is a finite reduced Gorenstein scheme in $\PP^{3}$ containing a subscheme $X$ with complement $Y$, we let
$$
\CCH_{X \cup Y = G} =\{ (X',Y',G') \in \CCH_{X} \times \CCH_{Y} \times \CCH_G | X' \cup Y' = G' \hbox{ and } X' \cap Y' = 0 \}.
$$
be the incidence correspondence, and we ask when $\CCH_{X\cup Y = G}$ projects dominantly onto $\CH_{X}$ and $\CH_{Y}$ via $\CCH_{X}$ and $\CCH_{Y}$; in this
case we say that the correspondence is bi-dominant. Of course for this to happen, the family $\CCH_{G}$ must have dimension
at least $\max(3d,3e)$. We will show that no bi-dominant correspondence is  possible unless the numbers $d,e$ are both $\leq 47$. 

Given $d,e$ and an appropriate Gorenstein Hilbert function, we search for an example of a reduced Gorenstein scheme $G\subset \P^{3}$ with the given Hilbert function such that $G$ contains a degree $d$ subscheme $X$ with generic Hilbert function whose complement
$Y$ also has generic Hilbert function. We do not know how to find such examples directly in characteristic 0.  It is perhaps the most surprising part of this paper that one can find examples of the type above by computer over a finite field; We can then lift the examples of pairs $X\subset G$ to characteristic zero. 
The success of the method in finite characteristic is based on the observation that, given integers $d,e$, a random polynomial of degree $d+e$ over a large finite field ``often'' has a factor of degree $e$.  This phenomenon is explored in Section~\ref{splittings}.

Set 
$$
\CCH_{X\subset G}=\{(X', G')\in \CCH_{X} \times \CCH_{G}\mid X'\subset G' \}.
$$ 
Near a triple $(X,Y,G)$ as above, the natural projection $\CCH_{X \cup Y = G}\to \CCH_{X\subset G}$ is an isomorphism, with inverse
defined from the family, over $\CCH_{X\subset G}$, of residual subschemes to $X$ in $G$, and the universal property of 
$\CCH_{X \cup Y = G}$.

We use a deformation-theoretic argument given in Section~\ref{defo}, together with machine computation, to test whether $\CCH_{X\subset G}$ is smooth at the pair
$(X,G)$ and projects dominantly onto $\CH_{X}$, and similarly for $Y$. If our example passes both these tests, it follows
that the incidence correspondence $\CCH_{X\cup Y = G}$ is bi-dominant.

If $\CCH_{X \cup Y =G}$ is bi-dominant then for any dense open set $V \subset \CCH_Y$ there exists a dense open set $U \subset \CCH_X$ such that each point $X' \in U$ can be directly Gorenstein-linked to a point $Y' \in V$. In particular, if a general point $Y' \in \CCH_Y$ is glicci, then a general point $X' \in \CCH_X$ is glicci as well.

Figure~\ref{Fig1-1} presents a graph (produced by the program Graphviz) of all bi-dominant correspondences. A node numbered $d$ is connected to a node numbered 
$e$ if, for a general set of $d$ points $X$, there exists a Gorenstein scheme of degree $d+e$ 
containing $X$ with complement $Y$ such that the scheme $\CCH_{X\cup Y =G}$ dominates both $\CCH_{X}$
and $\CCH_{Y}$. By the remark above, general set of $d$ points is glicci if $d$ lies in the connected component of $1$ in this graph.

%\subsection*{34 Points?}
\medbreak
{\noindent\bf 34 Points?} The degree of the smallest collection of general points in $\PP^{3}$ that is still not known to be glicci is  34. Here is a possible attack on this case: a general set of $34$ points can be linked (using a five-dimensional family of Gorenstein schemes) to a five-dimensional family $\cF$ of sets of 34 points. On the other hand the schemes in $\cH_{34}$ that can be directly Gorenstein linked to 21 points form a subfamily of codimension only $3$. Hence it is plausible that the family $\cF$ meets this stratum. If this does indeed happen, then, since we know that a set of 21 general points is glicci, it would follow that a set of 34 general points is glicci. (One could also link back and forth between sets of 34 and 36 points; or some combination of both.) To make this work would seem to require a good compactification of $\CCH_{G}$, on which one could do intersection theory. 

\section{Split Polynomials over Finite Fields}\label{splittings} 

In this section we describe the  philosophy leading us to believe that the computations underlying this paper could be successful. In a nutshell:  $\QQ$-rational points on varieties over $\QQ$ are very hard to find, but $\FF_{p}$-rational points on varieties over a finite field $\FF_{p}$ are much more accessible. No new result is proven in this section, and none of the  results mentioned will be used in the rest of the paper.

By a result of Buchsbaum and Eisenbud~\cite{BE}, the homogeneous ideal of an arithmetically Gorenstein subscheme $G$ of codimension 3 in projective space has a minimal presentation matrix that is homogeneous and skew-symmetric. The degrees of the elements in this matrix (the \emph{degree matrix}) determine the Hilbert function of $G$. For the arithmetically Gorenstein schemes in an open dense subset of the Hilbert scheme, the degree matrix is also determined (up to permutation of the rows and columns) by the Hilbert function. 

Our method thus requires that we find an arithmetically Gorenstein scheme with given degree matrix that also contains  reduced components---and these must be sufficiently general so that the Hilbert function of the corresponding subscheme is ``generic''. We do this by choosing a \emph{random} skew symmetric matrix with appropriate degrees over a moderately large finite field. We then \emph{hope} that it contains
a reduced subscheme of the right degree; and we check whether it does by projecting (in a random direction) to a line and factoring the polynomial in 1 variable that corresponds to the subscheme of the line. Once we have such an example we can proceed with the computations of tangent spaces described in the next sections.

For instance, to show that a set of 20 points in $\PP^{3}$ is directly Gorenstein-linked to a set of 10 Gorenstein points, we will choose a random arithmetically Gorenstein, 0-dimensional scheme  $G\subset \PP^{3}$ such that the presentation matrix of $I_{G}$ is a $9\times 9$ skew-symmetric matrix of linear forms (this choice of a $9\times 9$ matrix of linear forms is determined by considerations to be described later). If $M$ is a sufficiently general matrix of this kind over a polynomial ring in four variables then the cokernel of $M$ will be the homogeneous ideal of a reduced 0-dimensional scheme of degree 30 in $\PP^{3}$. Now suppose that the ground field $k$ is a moderately large finite field, and we choose such a matrix randomly (say by choosing each coefficient of each linear form in the upper half of the matrix uniformly at random from $k$). What will be the chance that it contains a  subscheme of length 20 defined over $k$? That is, how many random examples should one expect to investigate before finding a good one? 

The answer was surprising to us: Taking $k$ to be a field with 10,007 elements, and making 10,000 random trials, we found that the desired subscheme occurs in 3868 examples---about 38\% of the time. For the worst case needed for this paper, where the Gorenstein scheme has degree 90 and the desired subscheme has degree 45, the proportion is about 17\% in our experiments.

The proportion of Gorenstein schemes of degree 30 that are reduced and have a subscheme of degree 20 defined over $k$ turns out to be quite close to the proportion of polynomials of degree 30 that have a factor
of degree 20 over $k$. That proportion can be computed explicitly as a rational function in the size of the ground field; for $|k| = 10,007$ it is approximately .385426. 

Such phenomena are quite general (see \cite{Katz}, Theorem 9.4.4). Thus, to consider applications of this random search technique it is worthwhile to know something about the proportion of polynomials of degree $n$ in 1 variable over a finite field $\FF_{q}$ that are square-free and have a factor of given degree $k$.

This proportion can be computed explicitly (for small n and k).
Gauss showed that the number of irreducible monic polynomials in $\FF_q[x]$ of degree $\ell$ is
$$
N(\ell,q)= \frac{1}{\ell}\sum_{d|\ell} \mu(\ell/d) q^{d}
$$
where $\mu$ denotes the M\"obius function. Thus the number of square free polynomials is of degree $n$ is
$$
\sum_{\lambda \vdash n} \prod_{i=1}^r{N(\lambda_i,q)\choose t_i}
$$
where $t_i$ denotes the frequency of $\lambda_i$ in the partition $\lambda=(\lambda_1^{t_1},\lambda_2^{t_2}, \ldots,\lambda_r^{t_r}) $. 
(This number, rather amazingly, can also be written as  $q^n-q^{n-1}$; for a simple proof see \cite{sq-free}.)

The number of square free polynomials
of degree $n$ with a factor of degree $k$ is
$$
A(n,k,q)=\sum_{\lambda \vdash n \atop \hbox{\scriptsize with subpartition of size }k} \prod_{i=1}^r{N(\lambda_i,q)\choose t_i}. 
$$ 
For small $n$ and $k$ the polynomial in $q$ can be evaluated explicitly. For example
$$
A(6,3,q)=\frac{29}{80} q^{6}-\frac{11}{16} q^{5}+\frac{5}{16} q^{4}-\frac{5}{16} q^{3}+\frac{13}{40} q^{2}
$$
and
$$
A(30,20,q)/q^{30} \doteq .385481   - .550631q^{-1}+ O(q^{-2}).
$$

Since
$$
 \lim_{q \to \infty} N(\ell,q)/q^\ell = \frac{1}{\ell}
 $$
  the relative size of the contribution of a partition $\lambda$ converges to
$$
\lim_{q \to \infty} \prod_{i=1}^r{N(\lambda_i,q)\choose t_i} /q^n = \frac{1}{\prod_{i=1}^r t_i!\lambda_i^{t_i}}=|C_\lambda|/n! ,
$$
which is also the relative size of the conjugacy class $C_\lambda$ in the symmetric group $S_n$. Thus the sum
$$
p(n,k)=\sum_{\lambda \vdash n \atop \hbox{\scriptsize with subpartition of size }k} |C_\lambda|/n!
$$
can serve as an approximation for $A(n,k,q)/q^n$ for large $q$.

For fixed $k$, Cameron (unpublished) proved that the limit $\lim_{n \to \infty} p(n,k)$ exists and is positive. For example,
$$\lim_{n \to \infty} p(n,1) = 1-exp(-1) $$
was established by Montmort around 1708~\cite{Montmort}. 

Indeed,  over a finite field $\FF$ with $q$ elements the fraction of polynomials with a root in $\FF$ is about $63\%$ nearly independently of $q$ and $n$. Experimentally we find that
$$
A(n,1,q)/q^n \doteq   .632121  - .81606q^{-1}+O(q^{-2}).
$$

\section{Deformation Theory}\label{defo}

We are interested in pairs of schemes $X\subset G$ such that the projection
$$\CCH_{X\subset G}\to \CCH_{X}
$$
 is dominant, meaning geometrically that each small deformation of $X$ is still contained in
 a small deformation of $G$ that is still arithmetically Gorenstein.
We will check this condition
by showing that the map of tangent spaces is surjective at some smooth point of
$\CCH_{X\subset G}$. 

We thus begin by recalling the construction of the tangent space to $\CCH_{X \subset G}$. Though the case of interest to us has to do with finite schemes in $\PP^{n}$, the (well-known) result is quite general:

\begin{lemma} \label{defoLemma}
Suppose that $X\subset G$ are closed subschemes of a scheme $Z$, and let
$$
T_{X/Z}=H^{0}{\mathcal Hom}_{Z}(\CI_{X},\CO_{X}), \quad T_{G/Z}=H^{0}\CHom_{Z}(\CI_{G}, \CO_{G})
$$
 be the tangent spaces to the functors of embedded
flat deformations of $X$ and $G$ in $Z$.  
The functor of pairs of embedded flat deformations of $X$ and $G$ in $Z$ that preserve the inclusion relation $X\subset G$ has Zariski tangent space $T_{(X\subset G)/Z}$ at $X\subset G$ equal to $H^{0}\CT_{(X\subset G)/Z}$,
where $\CT:=\CT_{(X\subset G)/Z}$ is defined by the fibered product diagram
$$
\begin{diagram} 
\CHom_{Z}(\CI_{G},\CO_{G})&\rTo&\CHom_{Z}(\CI_{G},\CO_{X})\\
\uTo&&\uTo\\
\CT &\rTo&\CHom_{Z}(\CI_{X},\CO_{X}).
\end{diagram}
$$
In particular, if the restriction map $H^{0}\CHom(\CI_{X},\CO_{X}) \to H^{0}\CHom(\CI_{G},\CO_{X})$
is an isomorphism then $T_{(X\subset G)/Z} \cong T_{G/Z}$.
\end{lemma} 

\begin{proof} (See also \cite{Hartshorne2010} Ex. 6.8.) It suffices to prove the lemma in the affine case, so we suppose that
$Z = \Spec R$ and that  $X\subset G$ are defined by ideals $I\supset J$ in $R$.
The first order deformation of $I$ corresponding to a homomorphism $\phi: I\to S/I$ is the ideal
$$
I_{\phi}:=\{i+\e\phi(i)\mid i\in I\} + \e I \subset R[\e]/(\e^{2}),
$$
and similarly for $J_{\psi}$, so we have
$$
\CT = \{(\psi:J\to R/J, \phi:I\to R/I) \mid \psi(j) \equiv \phi_(j)(\hbox{mod } I) \hbox{ for all }j\in J\}.
$$
  If $\psi(j) \equiv \phi_(j)(\hbox{mod } I)$ for all $j\in J$ then 
every element $j+\e\psi(j)$ is obviously in $I_{\phi}$. Conversely, if $j+\e\psi(j) = i+\e\phi(i)+\e i'$
with $i'\in I$, then  $i=j$, so $\psi(j)=\phi(j)+i'$. This proves that $\CT$ is the fibered product. The last statement of the Lemma is an immediate consequence.
\end{proof}

Recall from Section~\ref{section:sketch} that if $X$ is reduced and has generic Hilbert function, then
$\CCH_X$ and the Hilbert scheme $\CH_X$ coincide in a neighborhood of $X$. 
Moreover,  $\CCH_X$ and $\CH_X$ are irreducible.

\begin{theorem}\label{dominance} Let $G\subset \PP^{n}$ be a finite scheme
such that cone over $G$ is a smooth point on $\CCH_{G}$. 
Suppose that $X\subset G$ is a union of some of the components of $G$ that are reduced, and that $X$ has generic Hilbert function. Let $d = \deg X$. If
$$
\dim_{k}\Hom_{S}(I_{G}, S_{G})_{0} - \dim_{k}\Hom_{S}(I_{G}, I_{X}/I_{G})_{0} = nd,
$$ 
then 
the projection map $\CCH_{ X\subset G}\to \CH_{ X}$ is dominant.
\end{theorem}

\begin{proof} 
Consider the diagram with exact row
$$\begin{diagram}
0&\rTo&\Hom_{S}(I_G,I_{X}/I_{G})&\rTo&\Hom_{S}(I_{G}, S_{G})&\rTo& \Hom_{S}(I_{G}, S_{X})\\
&&&&&&\uTo_{\phi}\\
&&&&&&\Hom_{S}(I_{X}, S_{X}).
\end{diagram}
$$
We begin by computing the dimension of $\Hom_{S}(I_{X}, S_{X})_{0}$, the degree zero part of $\Hom_{S}(I_{X}, S_{X})$. We may
interpret this space as the space of first-order infinitesimal deformations
of $X$ as a cone---that is, as the tangent space to $\CCH_{X}$ at the point $X$.
The computation of this space commutes with base change, and since we have assumed that
the ground field $k$ is perfect, the base change to $\overline k$ remains reduced. Thus to compute the dimension of
$\Hom_{S}(I_{X}, S_{X})_{0}$ we may assume that $X$ consists of $d$ distinct $k$-rational points. 

The sheaf $\CHom_{S}(\CI_{X}, \CO_{X})$ is the sheafification of
$\Hom_{S}(I_{X}, S_{X})$, so there is a natural map
$$
\alpha: \Hom_{S}(I_{X}, S_{X})_{0}\to H^{0}(\CHom_{S}(\CI_{X}, \CO_{X})).
$$
The source of $\alpha$ may be identified with the tangent spce to Hilbert scheme of cones
near $X$, while the target may be identified with the Hilbert scheme of collections of points near $X$,
and $\alpha$ is the map induced by forgetting the cone structure. 
Since, by assumption, $X$ has generic Hilbert function, these Hilbert schemes coincide, and $\alpha$ is
an isomorphism.
Thus $\dim \Hom_{S}(I_{X}, S_{X})_{0} = nd$. 

%Since any motion of a set of distinct points is
%scheme-theoretically flat, $nd$ is also the dimension of the 
%$H^{0}(\CHom_{S}(\CI_{X}, \CO_{X}))$, the tangent space at $X$ to the Hilbert scheme of $d$-tuples of points in $\PP^{n}$. Because $S_{X}$ is reduced, $\Hom_{S}(I_{X}, S_{X})$ has depth 1. It follows that $\alpha$ is a monomorphism and, since the dimensions of the source and target of $\alpha$ are the same,
%an isomorphism.

Though the map $\phi: \Hom_{S}(I_{X}, S_{X})\to \Hom_{S}(I_{G}, S_{X})$ in the diagram above
is generally not an isomorphism, we will next show that it induces an isomorphism between
the components of degree 0.
Using the fact that $S_{X}$ is reduced, so that $Hom_{S}(I_{G}, S_{X})$ has depth $\geq 1$, 
we see that 
the natural map
$$
\beta: \Hom_{S}(I_{G}, S_{X})_{0}\to H^{0}\CHom(\CI_{G}, \CO_{X})
$$
is an injection. On the other hand, any section of $\CHom(\CI_{G}, \CO_{X})$
is supported on $X$. Because $X$ is a union of the components of $G$, this implies that
$$
H^{0}\CHom(\CI_{G}, \CO_{X}) = H^{0}\CHom(\CI_{X}, \CO_{X}).
$$
Together with the equality $Hom_{S}(I_{X}, S_{X})_{0} = H^{0}\Hom(\CI_{X}, \CO_{X})$, this implies that the map
$$
\phi_{0}: \Hom_{S}(I_{X}, S_{X})_{0}\to \Hom_{S}(I_{G}, S_{X})_{0}
$$
is an isomorphism, as claimed.

Let $\cone X\subset \cone G\subset \AA^{n+1}$ be the cones over $X$ and $G$ respectively.
We may apply Lemma~\ref{defoLemma}, which tells us that the space of first-order 
deformations of the pair $\cone X\subset \cone G$  is the fibered product of $\Hom_{S}(I_{X}, S_{X})$ and $\Hom_{S}(I_{G}, S_{G})$ over $\Hom_{S}(I_{G}, S_{X})$. Since we wish to look only at deformations as cones, we take the degree 0 parts of these spaces, and we see that the tangent space to $\CCH_{X\subset G}$ is
the fibered product of $\Hom_{S}(I_{X}, S_{X})_{0}$ and $\Hom_{S}(I_{G}, S_{G})_{0}$ over $\Hom_{S}(I_{G}, S_{X})_{0}$. Since $\phi_{0}$ is an isomorphism, 
the tangent space to $\CCH_{X\subset G}$ is isomorphic, via the projection, to
$\Hom_{S}(I_{G}, S_{G})_{0}$, the tangent space to $\CCH_{G}$ at $G$.

Since $\cone X$ consists of a subset of the irreducible components of $\cone G$, and $X$ has generic Hilbert function, it follows that
the map $\CCH_{X\subset G}\to \CCH_{G}$ is surjective. Since $\CCH_{G}$ is smooth
at $G$, and the map of tangent spaces is an isomorphism, it follows that 
$\CCH_{X\subset G}$ is smooth at the pair $(X \subset G)$.

To prove that the other projection map $\CCH_{X\subset G} \to \CCH_{X}$ is dominant it now suffices to show that the map on tangent spaces
$$
T_{X\subset G} \CCH_{X\subset G}= T_{G}\CCH_{G} \to T_{X}\CCH_{X} =\Hom_{S}(I_{G}, S_{X})_{0}
$$
is onto or, equivalently, that
the right hand map in the sequence 
$$
0\to \Hom_{S}(I_G,I_{X}/I_{G})_{0}\to \Hom_{S}(I_{G}, S_{G})_{0}\to \Hom_{S}(I_{G}, S_{X})_{0}
$$
is surjective. Since the right hand vector space has dimension $nd$, this follows from our hypothesis on dimensions.
\end{proof}

\begin{corollary}\label{cordom} Let $G\subset \PP^{n}$ be a finite scheme.
Suppose that $X\subset G$ is a union of some of the components of $G$ that are reduced, and that $X$ has generic Hilbert function. Let $d = \deg X$. 
If
$$
\dim_{k}\Hom_{S}(I_{G}, I_{X}/I_{G})_{0} = \dim_G \CCH_G- nd,
$$ 
then $\CCH_G$ and $\CCH_{X\subset G}$ are smooth in $G$ respectively $(X \subset G)$
and the projection map $\CCH_{ X\subset G}\to \CCH_{ X}$ is dominant.
\end{corollary}

\begin{proof} Since 
$$
\begin{aligned}
\dim_G \CCH_G &\le \dim T_G{\CCH_G}= \dim \Hom_{S}(I_{G}, S_{G})_{0} \\
&\le \dim \Hom_{S}(I_{G}, I_{X}/I_{G})_{0}+ \dim \Hom_{S}(I_{G}, S_{X})_{0} \\
&=\dim \Hom_{S}(I_{G}, I_{X}/I_{G})_{0}+nd
\end{aligned}
$$
equality holds by our assumption, and $\CCH_G$ is smooth at $G$. Now the Theorem applies.
\end{proof}

\section{Computational Approach}\label{comp}

To classify the possible bi-dominant direct Gorenstein linkage correspondences we make use of h-vectors \cite{Stanley1978}, which are defined as follows. See \cite{HSS} for further details.

Let $R=S/I$ be a homogeneous Cohen-Macaulay factor ring of a polynomial ring S with $\dim R= k$.
The \emph{h-vector} of $R$ (or of $X$ in case $I=I_X$) is defined to be the $k$-th difference of the Hilbert function of $R$. If $\ell_1,\ldots, \ell_k$ is an R-sequence of linear forms, then $R/( \ell_1,\ldots,\ell_k)$ is artinian, and its Hilbert function is equal to the h-vector of $R$. It follows in particular that the  h-vector consists of a finite sequence of positive integers followed by zeros. Over a small field such an 
$R$-sequence may not exist, but the h-vector does not change under extension of scalars, so the conclusion remains true.  We often specify an h-vector by giving just the list of non-zero values.We can make a similar construction for any Cohen-Macaulay module.
If $X$ is a finite scheme then the sum of the terms in the h-vector is the degree of $X$. The h-vector of a Gorenstein ideal is symmetric.

From the definition it follows at once that the h-vector of a set of points $X$ with general Hilbert function 
is equal to the Hilbert function of the polynomial ring in 3 variables except (possibly) for the last nonzero term, and thus has the form 
$$
1,3,6,..., {s+1\choose 2}, a, 0,\dots \hbox{ with $0\leq a\leq {s+2\choose 2}$.}
$$
where $s$ is the least degree of a surface containing $X$. For example the h-vector of a general collection of $21$ points in $\PP^3$ is $\{1,3,6,10,1\}$.  

We will make use of the observation that if $I_X$ and $I_Y$ in $S$ are directly Gorenstein linked via $I_G$ then the h-vector of $X$, plus some shift of the reverse of the h-vector of $Y$, is equal to the h-vector of $G$ (see~\cite{HSS}, 2.14.)
One way to see this is to reduce all three of $I_{X},I_{Y}, I_{G}$ modulo a general linear form. 
The relation of linkage is preserved, and $\omega_{X} \cong I_{Y}/I_{G}$ (up to a shift). Since the Hilbert function of $\omega_{X}$ is the reverse of the Hilbert function of $S/I_{X}$, this gives the desired relation.

For example our computations show that a general collection of 21 points in $\P^{3}$, with h-vector
$\{1,3,6,10,1\}$ as above is directly Gorenstein linked to a collection of 9 points with general Hilbert function, and thus h-vector $\{1,3,5\}$. The Gorenstein ideal that links them will have h-vector
$\{1,3,6,10,6,3,1\}$, and we have:
\begin{align*}
\ 1\ 3\ 6\ 10\ &1\\
+\ \ \ \ \ \ \ \ \ \ \ \ \ \ \ \ \ &5\ 3\ 1\\
=\ \ 1\ 3\ 6\ 10\ &6\ 3\ 1
\end{align*}
The additivity of h-vectors in zero-dimensional Gorenstein liaison can be traced back to Macaulay (see~\cite{Macaulay1913}, p. 112)  as Tony Iarrobino pointed out to us.

We return to the problem of linking general sets of points. Let $\CH_d \subset Hilb_d(\PP^3)$ denote the irreducible component whose general point corresponds to a collection of $d$ distinct points.
Consider an (irreducible) direct  Gorenstein linkage correspondence
$$
\CCH_{X\cup Y = G} \to \CH_d\times \CH_e
$$
Recall that $\CCH_{X \cup Y = G}$ is said to be bi-dominant if it dominates both $\CH_d$ and $\CH_e$.
The h-vector of a general point $G =X\cup Y$ of a bi-dominant correspondence is rather special. Most of the following Proposition can be found in \cite{HSS}, 7.2; we repeat it for the reader's convenience.

\begin{proposition}\label{h-vectors} Let $G=X \cup Y $ be a general point of a bi-dominant correspondence. Then the h-vector of $G$ is one of the following:

\begin{itemize}
\item[I)] $\{1,3,6, \ldots, { s+1 \choose 2},  { s+1 \choose 2}+c,  { s+1 \choose 2},\ldots, 3, 1\}$ with $0\le c \le s+1$, or
\item[II)] $\{1,3,6, \ldots, { s+1 \choose 2},  { s+1 \choose 2}+c,  { s+1 \choose 2}+c, { s+1 \choose 2},\ldots, 3, 1\}$ with $0\le c \le s+1$
\end{itemize}
\end{proposition}

\begin{proof}  We may assume that $d \ge e$. A collection of $d$ general points has generic Hilbert function. Thus the h-vector of $X$ has the shape
$h_X=\{1,3, \ldots, {t+1 \choose 2}, a\}$, where $d={ t+1 \choose 3} +a$ is the unique expression with $0 \le a < { t+1 \choose 2}$. Similarly, we have
$h_{Y}=\{1,3, \ldots, {t'+1 \choose 2}, a'\}$, where $e={ t'+1 \choose 3} +a'$. On the other hand  since $G$ is arithmetically Gorenstein,  the h-vector of $G$ is symmetric.
As explained above, the difference  $h_G-h_X$ coincides, after  a suitable shift, with the h-vector of $Y$ read backwards.
Since the h-vector of $G$ coincides with the h-vector of $X$ up to position $t$ we  have
only the possibilities
\begin{enumerate}
\item $t=t'$ and $a=a'=0$ with $h_G= \{ 1, 3 , \ldots, {t+1 \choose 2}, {t+1 \choose 2}, \ldots,3,1\}$ of type $II)$ with $s=t-1$ and $c=s+1$
\item $t=t'+1$ and $a=a'=0$ with $h_G= \{ 1, 3 , \ldots, {s+2 \choose 2}, {s+1 \choose 2}, \ldots,3,1\}$ of type $I)$ with $s=t-1=t'$ and $c=s+1$
\item $t=t'+2$ and $a+a'={ s+1 \choose 2}$ with $h_G= \{ 1, 3 , \ldots, {s+2 \choose 2}, {s+1 \choose 2}, \ldots,3,1\}$ of type $I)$ with $s=t-1=t'+1$ and $c=s+1$
\item  $s=t=t'$ and $h_G= \{ 1, 3 , \ldots, {s+1 \choose 2}, a+a',  {s+1 \choose 2}, \ldots,3,1\}$ of type $I)$
\item  $s=t=t'$, $a=a' $ with $h_G= \{ 1, 3 , \ldots, {s+1 \choose 2}, a,a, {s+1 \choose 2}, \ldots,3,1\}$ of type $II)$

\end{enumerate}
By Stanley's theorem  \cite{Stanley1978} 4.2  the difference function of the first half of $h_G$ is nonnegative, so we have $a+a' \ge {s+1 \choose 2}$ and
$a=a' \ge {s+1 \choose 2}$ respectively in the last two cases.
\end{proof}

% Let $\CCH_{h}$ be the Hilbert scheme  of cones over arithmetically Gorenstein codimension 3 schemes
% in $\PP^{3}$ with h-vector $h$. Note that if $G$ is a scheme corresponding to a point of $\CCH_{h}$, then $\CCH_{h} = \CCH_{G}$
% since these schemes are irreducible. 
 
\begin{proposition} Let $G$ be an arithmetically Gorenstein set of points in $\PP^{3}$ with h-vector $h$ be as in Proposition~\ref{h-vectors},
and let $g(h)$ be the dimension of $\CCH_{G}$.
\begin{itemize}
\item In case I, $g(h)=4s(s+1)+4c-1$.
\item In case II, $g(h)=\frac{9}{2} s(s+1)+\frac{1}{2}c(c+13)-cs -1$.
\end{itemize}
\end{proposition}

\begin{proof} Case I) is type 3 of \cite{HSS} 7.2. Case II) is proved analogously, by induction on $s$, using  \cite{HSS} 5.3 starting with the cases of h-vectors $\{1,1,1,1\}$ and $\{1,2,2,1\}$. Note that when $c\ge 2$,
Case II) will at some point reduce to Type 2 of \cite{HSS} 7.2.
\end{proof}

\begin{corollary}
There are only finitely many bi-dominant correspondences.
\end{corollary}

\begin{proof} If $\CCH_{h}\to \CCH_{d}$ is dominant then
we must have $g(h)\geq \dim \CH_d = 3d.$ But examining the Hilbert functions in the 
different cases we find that $d\geq {s+2\choose 3}$, which is cubic in $s$,
while the functions $g(h)$ are quadratic in $s$, so the inequality cannot hold for large $s$.
Calculation shows that $s\le 5$, and that $d=47$ is the maximal degree possible \cite{ES-glicci}. (See also \cite{HSS} 7.2, 7.3.)
\end{proof}

\begin{theorem} The bi-dominant correspondences are precisely those indictated in Figure \ref{Fig1-1}. 
\end{theorem}

\begin{proof} To prove existence of a bi-dominant correspondence it suffices to find a smooth point $(X,Y) \in \CCH_{X\cup Y=G}$  and to verify that both maps on tangent spaces
$$T_{X\subset G}\CCH_{X\subset G}   \to T_X \CH_d \hbox{ and }  T_{Y \subset G}\CCH_{Y\subset G}  \to T_Y \CH_e$$ are surjective. 
We test \cite{ES-glicci} each of the finitely many triples consisting of an h-vector $h=\{h_0, \ldots, h_n\}$ and integers $(d,e)$ satisfying $g(h) \ge \max(3d,3e)$ and $\sum h_i = d+e$  and subject to the condition that
$h$ can be expressed as the sum of the h-vector of a general set of $d$ points and the reverse of
the h-vector of a general set of $e$ points, as follows:

\begin{enumerate}
\item Using the probabilistic method of section 1, find  a pair $X \subset G$ over a finite field $\FF_p$. \item Let $Y$ be the scheme defined by $I_Y= I_G:I_X$. Test whether $G$ is reduced, and whether $X$ and $Y$ have generic Hilbert functions. 
\item Test whether
$$
\dim \Hom(I_G,I_X/I_G)_0 = \dim_G \CCH_G - 3d
$$
and
$$ 
\dim \Hom(I_G,I_Y/I_G)_0 = \dim_G \CCH_G - 3e.
$$ 
\end{enumerate}
If the example $X\subset G$ and $Y$ passes the tests in 2,3  then, by Corollary \ref{cordom}, $(X,Y) \in \CCH_{X\cup Y=G}$ is a point on a bi-dominant correspondence over $\FF_{p}$. 
Since we may regard our example as the reduction mod $p$ of an example defined over some number field, this  shows the existence of a bi-dominant correspondence in characteristic zero.

There are nine pairs $(d,h)$, involving six different h-vectors, where the procedure above did not, 
in our experiments, lead to a proof of bi-dominance. They are given in the following table: \medskip

\noindent
\begin{tabular}{|r|l|l|}\hline
degrees & h-vector& Buchsbaum-Eisenbud matrix \cr \hline
7 & \{1,3,3,3,1\} & {\scriptsize$S^2(-3)\oplus S^3(-6) \to S^2(-5)\oplus S^3(-2)$}\cr
7 & \{1,3,3,3,3,1\} & {\scriptsize$S^2(-3)\oplus S^3(-5) \to S^2(-4)\oplus S^3(-2)$}\cr
13,14,15 &  \{1,3,6,6,6,3,1\}& {\scriptsize$S^3(-4)\oplus S^4(-6) \to S^3(-5)\oplus S^4(-3)$} \cr
16 & \{1,3,6,6,6,6,3,1\}& {\scriptsize$S^3(-4)\oplus S^4(-7) \to S^3(-6)\oplus S^4(-3)$} \cr
17 & \{1,3,6,7,7,6,3,1\} & {\scriptsize$S(-4)\oplus S(-5)\oplus S^3(-7) \to S(-6)\oplus S(-5)\oplus S^3(-3)$}\cr
25,26 & \{1,3,6,10,10,10,6,3,1\} &{\scriptsize $S^4(-5)\oplus S^5(-7) \to S^4(-6)\oplus S^5(-4)$}\cr\hline
\end{tabular} 
\medskip

\noindent
It remains to show
that,  in these numerical cases, there really is \emph{no} bi-dominant family. 
In each of these cases the Buchsbaum-Eisenbud matrix (the skew-symmetric presentation matrix of the $I_{G}$) has a relatively large block of zeroes, since the maps between the first summands of the free modules shown in the table is zero for degree reasons.
 (In case $d=17$ the map between the first two summands is zero, as the matrix is skew symmetric). Thus among the pfaffians of this matrix are the minors of an $n\times n+1$ matrix, for a certain value of $n$. These minors generate the ideal of an arithmetically Cohen-Macaulay (ACM) curve. In the given cases, the general such curve will be smooth. Thus, in these cases, the Gorenstein points lie on smooth ACM curves of degree $c$, the maximal integer in the h-vector of $G$ (so $c  \in \{3,6,7,10\}$.)
For example if $c=3$ there are 7 points, but a twisted cubic curve can contain at most 6 general points.
More generally, for a curve $C$ moving in its Hilbert scheme $\CH_C$ to contain $d$ general points we must have $2d \le \dim \CH_C$.
In all cases listed above, $\dim \CH_C=4c$
and $2d \le 4c$ is not satisfied.
\end{proof}

The method discussed above can be used to show more generally that having certain h-vectors forces a zero-dimensional Gorenstein scheme to be a divisor on an ACM curve. Here is a special case:

\begin{proposition}
If $Z$ is a zero-dimensional AG scheme  with h-vector $h$ of type I with $c=0$ or type II with $c=0,1$ in in Proposition \ref{h-vectors}, 
then $Z$ is a divisor in a class of the form $mH-K$ on some ACM curve whose h-vector is the first half of $h$.
\end{proposition}

We sketch an alternative proof:

\begin{proof}For the cases with $c=0$ this is a consequence of \cite{HSS} 3.4 (c). For type II with $c=1$ we use induction on $s$.
Using results 5.3 and 5.5 of \cite{HSS} we compare $h$ to the h-vector $h'$ defined there, which is the same thing with $s$ replaced by $s-1$, and we compute that the dimension fo $\CCH_{h}$ is equal to the dimension of the family of those $Z\sim mH-K$ on a $C$, computed as $\dim \CH_{h_{C}}+\dim_{C}|mH-K|$. The induction starts with $h= \{1,3,4,4,3,1\}$, where the corresponding scheme $Z$ is a complete intersection of type $(2,2,3)$ and the result is obvious.
\end{proof}

\begin{remark} In some cases the projection $\CCH_{X \subset G} \to \CH_X$ is finite.  This happens for the following degrees $d$ and h-vectors of $G$. \medskip

\begin{tabular}{|r|l|}\hline
degree & h-vector\cr \hline
7  & \{1,3,3,1\} \cr
17 & \{1,3,6,7,6,3,1\} \cr
21 & \{1,3,6,10,6,3,1\} \cr
25 & \{ 1,3,6,10,10,6,3,1\} \cr
29 &  \{1,3,6,10,12,10,6,3,1\} \cr
32&  \{1,3,6,10,12,12,10,6,3,1\}  \cr
33 & \{1,3,6,10,15,10,6,3,1\} \cr
38 & \{1,3,6,10,15,15,10,6,3,1\} \cr
45 & \{1,3,6,10,15,19,15,10,6,3,1\} \cr\hline
\end{tabular} 
\medskip

\noindent
It would be interesting to compute the degree of the projection in these cases.
When $G$ is a complete intersection the projection is one-to-one.
\end{remark}

\begin{corollary}\label{main} A general collection of  $d$ points in $\PP^3$
over an algebraically closed field of characteristic zero is glicci if $1\leq d\leq 33$ or $d=37$ or 38.
\end{corollary}

\begin{proof} Since the correspondences are bi-dominant, a general collection of $d$ points will be Gorenstein linked to a general collection of degree $e$. Thus we may repeat, and the result follows, because these degrees form a connected component of the graph in Figure \ref{Fig1-1}.
\end{proof}

\section{Strict Gorenstein Linkage}

One way to obtain an arithmetically Gorenstein (AG) subscheme of any projective space $\PP^{n}$ is to take an ACM subscheme $S$ satisfying the condition, called $G_{1}$, of being Gorenstein in codimension 1, and a divisor $X$ on it that is linearly equivalent to $mH-K$, where $H$ is the hyperplane class and $K$ is the canonical divisor of $S$ (see \cite{KMMNP}, 5.4). A slight variation of this construction allows one to reduce the condition $G_{1}$ to $G_{0}$ (Gorenstein in codimension 0; see \cite{Hartshorne2007}, 3.3.) A direct linkage  using one of these AG schemes is called a \emph{strict} direct Gorenstein link, and the equivalence relation generated by these is called \emph{strict Gorenstein linkage} \cite{Hartshorne2007}.

Nearly all of the  proofs in the literature that certain classes of schemes are glicci use this more restrictive notion of Gorenstein linkage. (See~\cite{Migliore1998} for a survey, and~\cite{Hartshorne2001, Hartshorne2002, Hartshorne2007, HMP2001, HMN2008,HSS,KMMNP} for some of the results)---in fact the one paper we are aware of that actually makes use of the general notion for this purpose is \cite{CDH} 7.1, which uses general Gorenstein linkages to show that any AG subscheme of $\PP^{n}$ is glicci. 

By contrast, some of the direct linkages established in this paper cannot be strict direct Gorenstein links.
We do not know whether such links can be achieved by a sequence of strict Gorenstein links; but one can show that if this is possible then some of the links must be to larger sets of points, and some of the intermediate sets of points must fail to be general.

\begin{proposition}\label{hartshorne1}
A general arithmetically Gorenstein scheme of 30 points in $\PP^{3}$ cannot be written as a divisor of the form $mH-K$
on any ACM curve  $C\subset \P^{3}$, where
$H$ is the hyperplane class and $K$ the canonical class of $C$. The linkages 20---10 and 21---9 in Figure 1 are not strict direct Gorenstein links.
\end{proposition}

\begin{proof} The h-vector of a Gorenstein scheme $Z$ of 30 points, of which 20 or 21 are general, is necessarily $h=\{1,3,6,10,6,3,1\}$. If $Z$ lay on an ACM curve in the class $mH-K$, then the h-vector of the curve would be $\{1,2,3,4\}$ (\cite{HSS}, 3.1). This is a curve of
degree 10 and genus 11. The Hilbert scheme of such curves has dimension 40, so such a curve can contain at most 20 general points. On the other hand, our Theorem~\ref{main theorem} shows that there are Gorenstein schemes with h-vector h containing 21 general points. In particular, the linkage 21---9 is not strict.

If the link 20---10 were a strict Gorenstein link, then a set of 20 general points $X$ would lie in an AG scheme of 30 points in the class $5H-K$ on a curve $C$ as above. Since the Hilbert scheme of $X$ and the incidence correspondence $\CH_{X\subset C}$ both have dimension 60, a general $X$ would be contained in a general, and thus smooth and integral, curve $C$. But the family of pairs
$Z\subset C$ of this type has dimension only 59 (\cite{HSS}, 6.8) and thus there is no such  $Z$ containing 20 general points.
\end{proof}

Some of the direct Gorenstein links in Figure 1 can be obtained by direct strict Gorenstein links (for example, the cases $d\leq 19$ are treated in \cite{Hartshorne2002}. However, for $d= 20, \ 24\leq d\leq 33$, and $d = 37, 38$ this is not the case.

\bigskip

\vbox{\noindent Author Addresses:\par
\smallskip
\noindent{David Eisenbud}\par
\noindent{Department of Mathematics, University of California, Berkeley,
Berkeley CA 94720}\par
\noindent{eisenbud@math.berkeley.edu}\par
\smallskip
\noindent{Robin Hartshorne}\par
\noindent{Department of Mathematics, University of California, Berkeley,
Berkeley CA 94720}\par
\noindent{robin@math.berkeley.edu}\par
\smallskip
\noindent{Frank-Olaf Schreyer}\par
\noindent{Mathematik und Informatik, Universit\"at des Saarlandes, Campus E2 4, 
D-66123 Saarbr\"ucken, Germany}\par
\noindent{schreyer@math.uni-sb.de}\par
}

\end{document}